\documentclass{article}

\usepackage[english]{babel}

\usepackage[letterpaper,top=2cm,bottom=2cm,left=3cm,right=3cm,marginparwidth=1.75cm]{geometry}

\usepackage{amsmath}  %
\usepackage{graphicx} %
\usepackage{amsthm}
\usepackage{amssymb}
\usepackage{bm}
\usepackage{graphicx}
\usepackage{subcaption}
\usepackage{blindtext}
\usepackage{url}
\usepackage{algorithm}
\usepackage{algorithmic}
\usepackage{xcolor}
\usepackage[final]{hyperref}
\usepackage{authblk}

\DeclareMathOperator*{\argmin}{arg\,min}
\newcommand{\calD}{\mathcal{D}}
\newcommand{\calG}{\mathcal{G}}
\newcommand{\calR}{\mathcal{R}}
\newcommand{\calP}{\mathcal{P}}

\newcommand{\R}{\mathbb{R}}
\newcommand{\cc}{\mathrm{c}}

\newcommand{\wass}{{{\mathcal W}_p}}

\newcommand{\pG}{\mathcal{P_G}}
\newcommand{\rd}{\mathrm{d}}

\theoremstyle{plain}
\newtheorem{theorem}{Theorem}[section]

\theoremstyle{definition}
\newtheorem{definition}[theorem]{Definition}

\theoremstyle{remark}

\title{Least-Squares Problem Over Probability Measure Space}

\author[1]{Qin Li}
\author[2]{Li Wang}
\author[3]{Yunan Yang}
\affil[1]{Department of Mathematics, University of Wisconsin--Madison}
\affil[2]{School of Mathematics, University of Minnesota Twin Cities}
\affil[3]{Department of Mathematics, Cornell University}

\date{}                     %
\begin{document}
\maketitle

\begin{abstract}
In this work, we investigate the variational problem 
\[
\rho_x^\ast = \argmin_{\rho_x} \calD(\calG_\#\rho_x, \rho_y)\,,
\]
where \(\calD\) quantifies the difference between two probability measures, and $\calG$ is a forward operator that maps a variable $x$ to $y=\calG(x)$. This problem can be regarded as an analogue of its counterpart in linear spaces (e.g., Euclidean spaces), $\argmin_x \|\calG(x) - y\|^2$. Similar to how the choice of norm \(\|\cdot\|\) influences the optimizer in \(\R^d\) or other linear spaces, the minimizer in the probabilistic variational problem also depends on the choice of \(\calD\). Our findings reveal that using a \(\phi\)-divergence for \(\calD\) leads to the recovery of a conditional distribution of \(\rho_y\), while employing the Wasserstein distance results in the recovery of a marginal distribution.

\end{abstract}

\section{Introduction}
We study the variational problem  
\begin{equation}\label{eqn:main_problem}  
\rho_x^\ast = \argmin_{\rho_x} \calD(\calG\#\rho_x, \rho_y)\,,\footnote{Since $(\calP(\mathbb{R}^m)$ is an infinite-dimensional space, the variational problem may not always have a solution. Throughout this note, we assume the problem is well-posed and that a solution exists.}
\end{equation}  
where \(\rho_x \in \calP(\Theta)\) is a probability measure over the variable \(x \in \Theta \subseteq \mathbb{R}^m\), with \(\calP(\Theta)\) denoting the collection of all probability measures on the domain \(\Theta\). The given data \(\rho_y \in \calP(\mathbb{R}^n)\) is a probability measure over the variable \(y \in \mathbb{R}^n\). The map \(\calG\#\) represents the push-forward operator, which maps a probability measure in \(\calP(\Theta)\) to a measure in \(\calP(\mathbb{R}^n)\) via the forward map: \(\calG : \Theta \subseteq \mathbb{R}^m \to \mathbb{R}^n\).  

This problem seeks a probability measure \(\rho_x\) such that, when pushed forward by \(\calG\), it matches \(\rho_y\) in an optimal sense. The notion of optimality is defined by a metric or divergence function \(\calD\).

Let \(\calR = \calG(\Theta) \subseteq \R^n\) denote the range of the map \(\calG\) over the domain \(\Theta\). In the case where \(\text{supp}(\rho_y) \subseteq \calR\), there exists at least one \(\rho_x\) such that \(\min_{\rho_x} \calD(\calG\#\rho_x, \rho_y) = 0\). Specifically, one can construct \(\rho_x^\ast = \calG^{-1}\#\rho_y\), where \(\calG^{-1}\) is understood as a left-inverse of $\calG$ that maps any \(y \in \calR\) to a point in its preimage.

The focus of this note, however, lies in the case where \(\text{supp}(\rho_y) \varsubsetneqq \calR\), implying that \(\rho_y(\mathbb{R}^n \setminus \calR) \neq 0\). In this scenario, no measure \(\rho_x\) satisfies \(\calG\#\rho_x = \rho_y\). Intuitively, this reflects the fact that \(\rho_y\) possesses a non-trivial mass outside the range \(\calR\), making a perfect match unattainable. Consequently, \(\min_{\rho_x} \calD(\calG\#\rho_x, \rho_y) > 0\).

This problem can be interpreted as an analogue of $\min_x \|\calG(x) - y\|^2$, the standard least-squares minimization, commonly posed in linear vector spaces such as Euclidean space \(\mathbb{R}^d\) or Hilbert function spaces \(L^2(\mathbb{R}^d)\). It is well known that different choices of the norm \(\|\cdot\|\) in this optimization problem emphasize different properties and, consequently, yield different minimizers.

In this note, we establish analogous results for Problem~\eqref{eqn:main_problem} posed over the space of probability distributions \(\calP(\Theta)\). Specifically, under mild conditions:  
\begin{itemize}
    \item For any choice of \(\phi\)-divergence, where \(\phi\) is convex, the reconstructed \(\calG\#\rho_x^\ast\) recovers the conditional distribution of \(\rho_y\) on \(\calR\).  
    \item For any choice of \(W_p\) (\(p \geq 1\)), the \(p\)-Wasserstein distance, \(\calG\#\rho_x^\ast\) recovers the marginal distribution of \(\rho_y\) projected onto \(\calR\).  
\end{itemize}

The above statement is not yet rigorous. The use of the terms ``conditional'' and ``marginal'' is intended for intuition and will be made precise later. When $\calG$ is linear and $\Theta$ is a linear space, \(\calR\) is also a linear space. In this case, conditional and marginal distributions are well-defined in the usual sense. However, when \(\calG\), and consequently \(\calR\), are nonlinear, more precise definitions are required.  Specifically:  
\begin{itemize}
    \item[--] The existence of the conditional distribution should be understood through the Theorem of Measure Disintegration.  
    \item[--] For the definition of the marginal distribution, we rely on a projection operator tailored to the context of nonlinear mappings.
\end{itemize}
We will provide a detailed discussion of these concepts in Sections~\ref{sec:f-div} and~\ref{sec:Wp}, respectively.

It is possible that these observations have appeared previously in the literature. However, to the best of the authors' knowledge, we have not found it well-documented. We welcome readers to provide suggestions or criticisms. We should also note that some of the results in simpler cases were reported in earlier work~\cite{li2024stochastic,li2023differential}. In particular, different recovery (conditional vs. marginal) for $\calG=\mathsf{A}$ as an overdetermined linear operator was reported in~\cite{li2024stochastic}. Furthermore, in~\cite{li2023differential}, the authors reported a gradient flow optimization algorithm using the kernel method when $\calD$ takes the form of Kullback--Leibler (KL) divergence.

In discussions with colleagues, we frequently encountered questions about how this formulation differs from Bayes' theorem and its related concepts in Bayesian inference and inverse problems. While we believe the differences are significant and self-evident, we dedicate Section~\ref{sec:bayes} to addressing this topic in detail. Readers already familiar with this distinction may skip this section without confusion. Section~\ref{sec:f-div} focuses on the general results for \(\phi\)-divergences, while Section~\ref{sec:Wp} addresses the counterpart for the Wasserstein distance. The proofs are concise enough to be included directly in the main text.

\section{Similarity and differences to Bayesian inversion}\label{sec:bayes}
Problem~\eqref{eqn:main_problem} is an optimization problem and should be viewed as a variant of the least-squares minimization problem  
\[
\min_x \|\calG(x) - y\|^2\,.
\]
It should not be interpreted as a formulation akin to Bayesian inference, which seeks a distribution that represents the probability that the unknown variable will be reconstructed. Although both formulations aim to find a probability distribution \(\rho_x \in \calP(\mathbb{R}^m)\) and solve an inverse problem based on given data, there are significant differences. Similar arguments have also been given in \cite{Estep}. From our perspective, the primary distinctions between these two types of inverse problems can be summarized in three key points:

\begin{itemize}
    \item The Bayesian formulation aims to reconstruct a single parameter \(x\). Since this reconstruction cannot be determined with certainty due to uncertainties in the data and/or models, a probabilistic argument is deployed as a necessarity to address these uncertainties.  

    In contrast, in Problem~\eqref{eqn:main_problem}, the unknown is the probability measure \(\rho_x\) itself, which can be determined with certainty if no noise exists in the given data distribution or model. Probabilistic argument is used simply due to the modeling assumption on $\rho_x$, not from data/model uncertainties. 
    
    \item In the Bayesian formulation, uncertainties arise from the forward map and/or measurement error. Notably, to construct the likelihood function -- one of the key components of the Bayesian framework -- explicit models for the uncertainties in the forward operator and the measurement error are required.  

    In contrast, we assume complete knowledge of the deterministic forward map \(\calG\) and no knowledge of the measurement error. If desired, one could frame a Bayesian inverse problem for Problem~\eqref{eqn:main_problem}, assuming uncertainties in \(\rho_y\). In this case, the unknown would become a probability measure over the space of probability measures, living in \(\calP(\calP(\Theta))\)~\cite{Marzouk}.
    
    \item The data \(y\) used in the Bayesian formulation typically consists of a finite number of realizations of the forward operator with measurement error, often denoted as \(\{y_k\}_{k=1}^N\), where each \(y_k \in \mathbb{R}^n\). According to Bayesian consistency theory~\cite{diaconis1986consistency}, as \(N \to \infty\) and more realizations of \(y\) are integrated, the posterior distribution \(\rho_x\) will converge to \(\delta_{x^\ast}\), the Dirac delta centered at the solution to the deterministic optimization problem \(\min_x \|\calG(x) - y\|^2\).  

    In contrast, the solution to Problem~\eqref{eqn:main_problem} will always remain a probability measure, since the primary source of randomness is inherent to \(x\), as captured by \(\rho_x\). This means that uncertainty is intrinsic: no matter how many realizations \(\{y_k\}\) are provided, the reconstruction does not collapse and remains a probability measure.
\end{itemize}
It is possible to interpret Bayes' formula as the solution to an optimization problem through the variational principle. Specifically, Bayes' formula can be recast as the solution to:
\[
    \rho_{\text{pos}} =  \argmin_{\rho_x } \bigg\{  \text{KL}(\rho_x || \rho_{\text{prior}}) + \int \phi(x,y) \rd\rho_x(x) \bigg\}\,,
\]
where $\rho_{\text{prior}}$ is the prior distribution, $\phi(x,y) = - \log L(y|x)$ is the negative log-likelihood function with \(L(y|x)\) being the likelihood that incorporates the forward map \(\calG\). \(\text{KL}\) denotes the KL divergence. The significance of this re-formulation is that it enables the Langevin Monte Carlo for Bayesian sampling. In contrast, our new formulation does not require any specific form for the energy functional, nor do we impose a rigid structure for \(\rho_x\).

\section{Reconstruction when $\calD$ is $\phi$-divergence}\label{sec:f-div}

This section presents results when \(\calD\) in~\eqref{eqn:main_problem} is the \(\phi\)-divergence. We rewrite~\eqref{eqn:main_problem} as follows:
\begin{equation}\label{eq:f_div_min}
\min_{\rho_x \in \calP(\Theta)} \calD_\phi (\calG{\#} \rho_x \| \rho_y)\,.
\end{equation}
The \(\phi\)-divergence between two probability measures \(P\) and \(Q\) is defined as
\[
D_\phi(P \| Q) = \int \phi\left(\frac{\rd P}{\rd Q}\right) \, \rd Q\,,
\]
where \(\frac{\rd P}{\rd Q}\) is the Radon–Nikodym derivative of \(P\) with respect to \(Q\), and \(\phi: \mathbb{R}^+ \to \mathbb{R}\) is a convex function such that \(\phi(1) = 0\). Common examples of \(\phi\)-divergences include the Kullback–Leibler (KL) divergence and the \(\chi^2\)-divergence.

Next, we define the conditional distribution of a measure \(\rho_y\) on the range \(\calR\).

\begin{definition}\label{def:conditional}
The conditional distribution of \(\rho_y\) on the range \(\calR\) is defined as follows:
\[
\rho^\cc_{y|\calR}(B) = \frac{\rho_y(B \cap \calR)}{\rho_y(\calR)}\,, \quad \forall\, \text{measurable set } B\,.
\]
Moreover, \(\rho_{y|\mathbb{R}^n \setminus \calR}^\cc\) is the conditional distribution of \(\rho_y\) on \(\mathbb{R}^n \setminus \calR\), defined in a similar manner.
\end{definition}

As a consequence, \(\rho^\cc_{y|\calR}\) is absolutely continuous with respect to \(\rho_y\), and we have:
\[
\frac{\rd\rho^\cc_{y|\calR}}{\rd\rho_{y}}(y) = \begin{cases}
\frac{1}{\rho_{y}(\calR)}\,, & \text{if } y \in \calR\,, \\
0\,, & \text{if } y \notin \calR\,.
\end{cases}
\]

We are now ready to present our first theorem:

\begin{theorem}\label{thm:phi_divergence}
Assume that the variational problem~\eqref{eq:f_div_min} admits a minimizer \(\rho_x^* \in \calP(\Theta)\). Then, we have
\[
\calG{\#} \rho_x^* = \rho^\cc_{y|\calR}\,.
\]
\end{theorem}

The proof of this theorem requires the use of the Measure Disintegration Theorem, which guarantees the existence of the conditional distribution. Specifically, for any given map \( T \) that maps from a probability space \( (X, \mathcal{B}_X, \mu) \) to a measurable space \( (Y, \mathcal{B}_Y) \), and defining \(\nu = T\# \mu\), the theorem states that for \(\nu\)-almost every \(y \in Y\), there exists a family of probability measures \(\{\mu_y : y \in Y\}\) on \( (X, \mathcal{B}_X) \) that satisfies: $\mu(B) = \int_Y \mu_y(B) \, \rd\nu(y)$ for any measurable set $B \in \mathcal{B}_X$. The collection \(\{\mu_y\}_y\) is called the disintegration of \(\mu\) with respect to \(T\). In our context, $X$ would be the whole of $\R^n$ and $\mu$ is our data $\rho_y$. We need to identify the appropriate measurable space and find the correct map \(T\), as detailed in the following proof.

\begin{proof}[Proof of Theorem~\ref{thm:phi_divergence}]
    Define a map \( T: \mathbb{R}^n \rightarrow \{0,1\} \) such that 
    \[
    z = T(y) = \begin{cases}
        1 \,, &y \in \calR\,,\\
        0 \,, & y \not\in \calR\,.\\
    \end{cases}
    \]
    By applying the Measure Disintegration Theorem~\cite{ambrosio2008gradient} to \( \rho_y \) based on the map \( T \), we obtain a discrete probability measure \( \nu \), with 
    \[
    \nu(1) = \rho_{y}(\calR), \quad \nu(0) = \rho_{y}(\mathbb{R}^n \setminus \calR),
    \]
    and the following disintegration of \( \rho_y \):
    \begin{equation}
       \rho_{y} = \nu(1) \, \rho_{y|\calR}^\cc + \nu(0)  \,  \rho_{y|\mathbb{R}^n \setminus \calR}^\cc\,,
    \end{equation}
    where \( \rho_{y|\calR}^\cc \) and \( \rho_{y|\mathbb{R}^n \setminus \calR}^\cc \) are the conditional distributions defined in Definition~\ref{def:conditional}. The Measure Disintegration Theorem further states that this disintegration is unique.
   
To show that \( \rho_{y|\calR}^\cc \) is the optimal solution, we rewrite the variational problem~\eqref{eq:f_div_min} as:
\begin{equation}\label{eq:f_div_min_2}
\min_{\rho_x \in \calP(\Theta)} \calD_\phi (\calG {\#} \rho_x || \rho_y) = \min_{\substack{\rho_y' \in \calP(\R^d) \\ \rho_y'({\R^n\setminus\calR})=0 }} \calD_\phi (\rho_y' || \rho_y)\,.
\end{equation}
This is true because $\{\calG \# \rho_x |\rho_x\in \calP(\Theta)\}   = \{\rho_y'\in \calP(\R^d): \rho_y' (\R^n \setminus \calR) = 0\}$. The ``$\subseteq$'' direction holds directly since $\calR = \calG(\Theta)$.  The ``$\supseteq$'' direction holds because, for a given $\rho_y'$, $\text{supp}(\rho_y')\subseteq \calR$ and by using the left inverse function of $\calG$, we can define a distribution $\rho_x' \in \calP(\Theta)$ such that $\calG \# \rho_x' = \rho_y'$.

Without loss of generality, we only examine $\rho_y'$ that is absolutely continuous with respect to $\rho_y$\footnote{Otherwise, $\calD_\phi(\rho_y' || \rho_y)$ achieves the maxmimum value, making such $\rho_y'$ irrelevant as we aim to find the minimum. The maximum value for this divergence is $\infty$ in the case of KL and $\chi^2$, and $1$ in the case of TV.}. By the definition of the \( \phi \)-divergence, we have
    \begin{eqnarray*}
    \calD_\phi(\rho_y' || \rho_y) &=& \int \phi\left(\frac{\rd\rho_y'}{\rd\rho_y}\right) \rd\rho_y\\
        &=& \nu(1) \int \phi\left(\frac{\rd\rho_y'}{\rd\rho_y}\right) \rd\rho_{y|\calR}^\cc +  \nu(0) \int \phi\left(\frac{\rd\rho_y'}{\rd\rho_y}\right) \rd\rho_{y|\mathbb{R}^n\setminus \calR}^\cc\,.
    \end{eqnarray*}
Since \( \rho_y'(\mathbb{R}^n \setminus \calR) = 0 \), the second term is a constant $\nu(0)\phi(0)$. Thus, we are left with:
\begin{eqnarray*}
    \calD_\phi(\rho_y' || \rho_y) &=& \nu(1) \int \phi\left(\frac{\rd\rho_y'}{\nu(1)  \rd\rho_{y|\calR}^\cc + \nu(0)  \rd\rho_{y|\mathbb{R}^n\setminus \calR}^\cc }\right) \rd\rho_{y|\calR}^\cc + \nu(0)\phi(0)\\
    &=& \nu(1) \int \phi\left(\frac{1}{\nu(1) }\frac{\rd\rho_y'}{ \rd\rho_{y|\calR}^\cc} \right) \rd\rho_{y|\calR}^\cc + \nu(0)\phi(0) \\
    &\geq &\nu(1) \phi\left(\frac{1}{\nu(1)}\right) + \nu(0)\phi(0)\,,
\end{eqnarray*} 
where in the last step we applied Jensen's inequality, leveraging the convexity of \( \phi \). The equality holds when \( \rho_y' = \rho_{y|\calR}^\cc \), completing the proof.
\end{proof}
Although the reconstructed $\calG\# \rho_x^*$ is expected to somewhat agree with $\rho_y$ on the range $\calR$, the fact that the mismatch between $\rho_y$ and  $\calG\# \rho_x^*$ on $\calR$ is merely a constant multiple is not entirely trivial. Jensen's inequality and the convexity of $\phi$ play a major role.

\section{Reconstruction when $\calD$ is $W_p$}\label{sec:Wp}

We now move on by setting \( \calD \) as a Wasserstein distance. This amounts to rewriting~\eqref{eqn:main_problem} as:
\begin{equation}\label{eq:W_min}
\inf_{\rho_x \in \calP(\Theta)} \wass (\calG{\#} \rho_x,  \rho_y^\delta),
\end{equation}
where \( \wass(\cdot, \cdot) \) is the \( p \)-Wasserstein distance. For any two probability measures \( \mu \) and \( \nu \), the \( p \)-Wasserstein distance is:
\begin{equation}\label{eq:wass_p}
\wass(\mu, \nu) = \left(\min_{\gamma \in \Gamma(\mu, \nu)} \int d^p(x, y) \,\mathrm{d}\gamma\right)^{1/p}\,, \quad p \geq 1\,,
\end{equation}
where \( d \) is a metric over \( \mathbb{R}^n \) and \( \Gamma(\mu,\nu) \) represents the set of all couplings between the two measures $\mu$ and $\nu$. The most common choice of $d$ is the Euclidean distance.

To characterize the minimizer of \eqref{eq:W_min}, we first define a projection operator \( \pG \) as follows. 
\begin{definition} \label{def:proj}
    Assume that the range \( \calR \) is compact. Define the projection operator \( \pG: \mathbb{R}^n \rightarrow \calR \) as 
    \[
    \pG(y^*) = \argmin_{y \in \calR} d(y,y^*)\,.
    \]
\end{definition}
\begin{theorem}
Assume \eqref{eq:W_min} admits a minimizer \( \rho_x^* \in \calP(\Theta) \), then we have
\[
\calG{\#} \rho_x^*  = \pG{\#} \rho_y\,.
\]
\end{theorem}

\begin{proof}
    By the definition of the \( \wass \) distance, we have 
    \begin{align} \label{p1}
        \wass(\calG{\#} \rho_x, \rho_y)^p = \int d(\tilde{y}, y)^p \, \rd \pi(\tilde{y}, y)\,,
    \end{align}
    where \( \pi \in \Gamma(\calG{\#} \rho_x, \rho_y) \) is the optimal coupling between \( \calG{\#} \rho_x \) and \( \rho_y \). From Definition~\ref{def:proj}, we can deduce that 
    \begin{align}
    \int d(\tilde{y}, y)^p \, \rd \pi(\tilde{y}, y) &\geq 
    \int d ( \pG(y), y)^p \, \rd \pi(\tilde{y}, y) \nonumber 
    \\ 
    & =  \int d(z, y)^p \, \rd \gamma(z, y)\quad \text{where }  \gamma = (\pG \otimes \mathbb{I}) \# (\rho_y \otimes \rho_y) \nonumber \\
    & \geq \wass(\pG{\#} \rho_y, \rho_y)^p \,. \label{p2}
    \end{align}
    Considering \( \pG{\#} \rho_y \in \calP(\calR) \), we can construct at least one distribution \( \rho_x \) such that \( \calG{\#}  \rho_x = \pG{\#} \rho_y \) using a left-inverse function of $\calG$.
    Finally, combining~\eqref{p1} and~\eqref{p2} leads to the target result. 
\end{proof}

\bibliographystyle{alpha}
\bibliography{sample}

\begin{thebibliography}{LOWY24}

\bibitem[AGS08]{ambrosio2008gradient}
Luigi Ambrosio, Nicola Gigli, and Giuseppe Savar{\'e}.
\newblock {\em Gradient flows: in metric spaces and in the space of probability
  measures}.
\newblock Springer Science \& Business Media, 2008.

\bibitem[BBE11]{Estep}
J.~Breidt, T.~Butler, and D.~Estep.
\newblock A measure-theoretic computational method for inverse sensitivity
  problems i: Method and analysis.
\newblock {\em SIAM Journal on Numerical Analysis}, 49(5):1836--1859, 2011.

\bibitem[DF86]{diaconis1986consistency}
Persi Diaconis and David Freedman.
\newblock {On the consistency of Bayes estimates}.
\newblock {\em The Annals of Statistics}, pages 1--26, 1986.

\bibitem[LOWY24]{li2024stochastic}
Qin Li, Maria Oprea, Li~Wang, and Yunan Yang.
\newblock {Stochastic Inverse Problem: stability, regularization and
  Wasserstein gradient flow}.
\newblock {\em arXiv preprint arXiv:2410.00229}, 2024.

\bibitem[LWY23]{li2023differential}
Qin Li, Li~Wang, and Yunan Yang.
\newblock Differential equation--constrained optimization with stochasticity.
\newblock {\em SIAM/ASA Journal on Uncertainty Quantification}, 11(2):491--515,
  2023.

\bibitem[Mar24]{Marzouk}
{Personal communication with Professor Youssef Marzouk}, 2024.

\end{thebibliography}

\end{document}